\documentclass[12pt]{amsart}
\usepackage{amsmath,amssymb,amsthm,mathtools}
\usepackage{fourier}
\usepackage{tikz}

\theoremstyle{plain}
\newtheorem{theorem}{Theorem}
\newtheorem{lemma}[theorem]{Lemma}
\newtheorem*{theorem*}{Theorem}

\begin{document}
\author{Iwan Praton}
\address{Franklin \& Marshall College} \email{ipraton@fandm.edu}
\title{Maximal Tilings with the Minimal Tile Property}
\date{}
\maketitle

\begin{abstract}
A tiling of the unit square is an MTP tiling if the smallest
tile can tile all the other tiles. We look at the function 
$f(n)=\max \sum s_i$, where $s_i$ is the side length
of the $i$th tile and the sum is taken over all MTP tilings
with $n$ tiles. If $n=k^2+3$, it was conjectured that
$f(k^2+3)=k+1/k$. We show that any tiling that violates
the conjecture must consist of at least three tile sizes
and has exactly one minimal tile. 
\end{abstract}

\section*{Introduction and notation}
Suppose we tile a unit square with small squares of side lengths
$s_1,\ldots,s_n$. (Thus the interiors of the small squares are
disjoint, and their total area is 1.) We define the function $f(n)=
\max \sum s_i$ where the maximum is taken over all such
tilings with $n$ square tiles. Erd\H{o}s and Soifer 
\cite{ErdosSoifer}
posed the problem of determining $f(n)$, although their 
question was more general in that
they considered packings of the unit squares, not just
tilings.  They also presented
conjectural values for $f(n)$, e.g., $f(k^2+1)=k$ and $f(k^2+3)
=k+1/k$. Erd\H{o}s offered \$50 for a proof or disproof,
so the problem appeared difficult.

In order to make the problem more computationally tractable, Alm 
\cite{Alm} 
introduced an extra condition. We say that a tiling $T$
of the unit square has the \emph{Minimal Tile Property} (MTP) if
every tile in $T$ can be tiled by the smallest tile. We thus have
the (hopefully) easier problem of determining $f_M(n)=
\max \sum s_i$, where the maximum is taken over all MTP
tilings of the unit square. 

If $n=k^2+3$ (where $k\geq 2$), the conjectured value of
$f(n)$ in 
\cite{ErdosSoifer}
is $k+1/k$, and this value is realized by an
MTP tiling: the standard $k\times k$ grid of small squares,
where one of the tiles is further divided into 4 smaller square
tiles. Thus a natural conjecture arises: $f_M(k^2+3)=k+1/k$.
The present paper studies this conjecture. In particular,
we show that if we want to look for counterexampes to the
conjecture, then we only need to consider tilings with 
exactly one minimal tile. 

We begin with some notations. All subsequent tilings are
MTP tilings of the unit square; the number of tiles
is $k^2+3$ for $k\geq 2$. If $T$ is a tiling with tiles
of side lengths $s_1,\ldots,s_n$, then we define 
the length of the tiling $\sigma(T)$ to be $\sum s_i$.
More generally, if $S$ is a collection of tiles, we
define $\sigma(S)$ to be the sum of the side lengths
of tiles in $S$.

We write the side length of the smallest tile as $1/a$;
note that MTP implies that $a$ is an integer. We usually
say that $T$ consists of $n_i$ tiles with side lengths
$m_i/a$, with $m_1< m_2<\cdots$. Note that $m_1=1$
and each $m_j$ is an integer (again by MTP). 

We also recall the parameter $\gamma(T)=\sum_{i<j}
(s_i-s_j)^2$ introduced in
\cite{StatonTyler}. 
We use a version that
seems more suitable for our purposes:
\[
\delta(T)=\sum_{i<j} n_in_j(m_i-m_j)^2 - \sum_i n_im_i^2.
\]

\begin{lemma}
If $\delta(T)\geq 0$, then $\sigma(T)<k+1/k$.
\end{lemma}

\begin{proof}
Recall from
\cite{StatonTyler}
that $\sigma(T)=\sqrt{k^2+3-\gamma(T)}$.
In our case,
\[
\gamma(T)=\sum_{i<j} n_in_j (m_i/a-m_j/a)^2
= \frac1{a^2}\sum_{i<j} n_in_j(m_i-m_j)^2.
\]

Suppose $\delta(T)\geq 0$. Then 
$\sum_{i<j} n_in_j(m_i-m_j)^2\geq \sum n_im_i^2$. Dividing
by $a^2$, we get $\gamma(T)\geq \sum n_i (m_i/a)^2 = 
\text{ total area of the tiles } = 1$. Therefore the total length 
of the tiling is
\[
\sigma(T)=\sqrt{k^2+3-\gamma(T)}\leq \sqrt{k^2+3-1}=
\sqrt{k^2+2}<k+1/k,
\]
as required.
\end{proof}

In our calculations of $\delta(T)$ we often encounter
an expression of the form $n_1(m-1)^2-m^2$. We note
here a positivity result about this expression. 

\begin{lemma}\label{sq}
If $n_1\geq 2$ and $m\geq 4$, then
$n_1(m-1)^2-m^2\geq n_1$.
\end{lemma}

\begin{proof}
 This is a straightforward calculation: $n_1(m-1)^2-m^2-n_1
=n_1\big((m-1)^2-1\big)-m^2\geq 2(m^2-2m)-m^2=m(m-4)\geq 0$.
\end{proof}

\section*{Tilings with two tile sizes} 
In this section we consider the case where there are only two
tile sizes. Thus our tiling $T$ has $n_1$ tiles of side length $1/a$ and
$n_2$ tiles of side length $m_2/a$, where $n_1+n_2=k^2+3$
and $n_1+n_2m_2^2=a^2$. We will show that $T$ can only
be maximal when $n_1=4$ and $m_2=2$, i.e., when $T$ is the
conjectured maximal tiling.

\begin{lemma}\label{n1}
In this case, $n_1\geq 2$.
\end{lemma}

\begin{proof}
Suppose $n_1=1$. Then the unique smallest tile cannot be 
on an edge of the unit square. Thus in the middle of the
unit square there is a vertical stack $S$
of tiles, including the smallest one, such that $\sigma(S)=1$.
Similarly there is another vertical stack $S'$ of tiles, say
on the edge of the unit square, with $\sigma(S')=1$. 
Note that $S'$ does not include the smallest tile. Thus
$x_2(m_2/a)=1$ for some integer $x_2$. For the stack $S$
we have $(1/a)+y_2(m_2/a)=1$ for some integer $y_2$.
Thus $x_2m_2=y_2m_2+1$, which is a contradiction.
\end{proof}

In our current case, the parameter $\delta(T)$ takes a
particularly simple form.
\[
\delta(T)=n_1n_2(m_2-1)^2-n_1-n_2m_2^2
= n_2[n_1(m_2-1)^2-m_2^2]-n_1.
\]

\begin{lemma}\label{l2}
$T$ is not maximal unless $n_1=4$ and $m_2=2$. 
\end{lemma}
\begin{proof}
If $m_2\geq 4$, then by Lemmas \ref{sq} and \ref{n1}, 
we have $\delta(T)\geq 0$ and hence $T$ is not maximal.
So we only need to look at $m_2=2$ and $m_2=3$.

Take the case $m_2=3$ first. Here $\delta(T)=
n_2(4n_1-9)-n_1 = n_1(4n_2-1)-9n_2$. If $n_1\geq 3$,
then $\delta(T)\geq 3(4n_2-1)-9n_2=3(n_2-1)\geq 0$,
so $T$ would not be maximal. Thus $n_1=2$, i.e.,
$T$ has two tiles of side length $1/a$ and $n_2$ tiles
of side length $3/a$. Since the total area is $1$,
we have $2+9n_2=a^2$, but this is impossible
modulo 9. Thus $T$ is not maximal when $m_2=3$.

Now take the case $m_2=2$. Here $\delta(T)=
n_2(n_1-4)-n_1=(n_1-4)(n_2-1)-4$.
There are many cases to consider.

Suppose $n_2=1$. Then $n_1=k^2+2$ and 
$n_1+4=a^2$, which implies that $k^2+6=a^2$.
Thus $(a+k)(a-k)=6$, an impossibility since $a$
and $k$ are both integers.

Suppose $n_2=2$. Similar considerations as
above leads to $(a+k)(a-k)=9$, which means
$a=5$ and $k=4$. Then $n_1=17$ and thus
$\delta(T)>0$. So $T$ is not maximal in this case.

Suppose finally that $n_2\geq 3$. Then $\delta(T)
=(n_1-4)(n_2-1)-4\geq 2n_1-12$. So $n_1\geq 6$
means $\delta(T)\geq 0$ and $T$ is not maximal.
We need to check each value of $n_1$ from $n_1=2$
to $n_1=5$. We know that $n_1+4n_2=a^2$,
so $n_1=2$ and $n_1=3$ are not possible
modulo 4. If $n_1=5$, then $n_2=k^2-2$. 
In this case,  $k=2$ implies $n_2=2$,
which by the previous paragraph
means that $T$ is not maximal. If $k\geq 3$
then $n_2\geq 7$, which implies that $\delta(T)>0$,
again showing that $T$ is not maximal. Thus $n_1=4$
as required.
\end{proof}

\section*{Tilings with three tile sizes}
In this section our tiling $T$ consist of three tile sizes:
$n_1$ tiles of side length $1/a$, $n_2$ tiles of side
length $m_2/a$, and $n_3$ tiles of side length $m_3/a$,
where $m_3>m_2>1$.
In this case Lemma \ref{n1} no longer applies.
In fact, it is possible to have an MTP tiling 
with a unique smallest tile if we have more than
just two tile sizes. 
So in this and subsequent sections we  assume
that $n_1\geq 2$.

The parameter $\delta(T)$ takes the form
\begin{align*}
\delta(T) &= A+B, \\
\text{where } A &= n_2[n_1(m_2-1)^2-m_2^2 + n_3(m_3-m_2)^2]\\
B &= n_3[n_1(m_3-1)^2-m_3^2]-n_1
\end{align*}

\begin{lemma}\label{l3}
Under the assumption that $T$ has three tile sizes
and $n_1\geq 2$, then $T$ is not maximal.
\end{lemma}
\begin{proof}
We use the formula for $\delta(T)$ above.
Suppose $m_2\geq 4$. Then $m_3\geq 5$.
Since we assume that $n_1\geq 2$,
Lemma \ref{sq} implies that $A, B\geq 0$, so in this
case $\delta(T)\geq 0$.

If $m_2=3$, then $m_3\geq 4$ so $B\geq 0$. Also,
$A=n_2[4n_1-9+n_3(m_3-2)^2]\geq n_2[8-9+1]=0$,
so again $\delta(T)\geq 0$.

If $m_2=2$ and $m_3\geq 4$, then $B\geq 0$ and
$A=n_2[n_1-4+n_3(m_3-2)^2]\geq n_2[n_1-4+4n_3]\geq 0$,
so $\delta(T)\geq 0$ once more.

It remains to investigate the case $m_2=2$, $m_3=3$
(again, with the assumption that $n_1\geq 2$). There are
unfortunately many case-by-case considerations.

Suppose $n_1\geq 3$. Then 
\begin{align*}
\delta(T)&=n_1(n_2+4n_3-1)+n_2n_3-4n_2-9n_3\\
&\geq 3(n_2+4n_3-1)+n_2n_3-4n_2-9n_3\\
&=(n_2+3)(n_3-1)\geq 0.
\end{align*}

Suppose now that $n_1=2$. In this case we can write $\delta(T)=
(n_2-1)(n_3-2)-4$. We will eliminate the values $n_2=1$, 
$n_2=2$, $n_3=1$, $n_3=2$, and $n_3=3$.

If $n_2=1$, then we have $2+4+9n_3=a^2$, or $a^2=6+9n_3$.
There are no solutions modulo 9.

If $n_2=2$, then $n_3=k^2-1$, so $n_3=3$ or $n_3\geq 8$. 
If $n_3\geq 8$ then $\delta(T)>0$. If $n_3=3$, then $a^2=
2+2\cdot 4 + 3\cdot 9 = 37$, an impossibility. 

If $n_3=1$, then we have $2+4n_2+9=a^2$, or $a^2=11+4n_2$.
This has no solutions modulo 4.

If $n_3=2$, then $n_2=k^2-1$ and $a^2=4n_2+20$. Thus 
$a$ is even, say $a=2b$. Then $b^2=n_2+5$, so $b^2=k^2+4$,
i.e., $(b-k)(b+k)=4$. This leads to $k=0$ or $k=3/2$, an
impossibility.

If $n_3=3$, then $n_2=k^2-2$, so $n_2=2$ or $n_2\geq 7$.
We have already eliminated the possibility of $n_1=2, n_2=2$,
and $n_3=3$ above. If $n_2\geq 7$, then $\delta(T)>0$. 

We conclude that $n_2\geq 3$ and $n_3\geq 4$. Then
$\delta(T)\geq 2\cdot 2 - 4=0$, so we are done.
\end{proof}

It seems possible to show, using \textit{ad hoc} methods,
that any tiling with exactly three tile sizes is not optimal.
But such methods do not seem to generalize to more
numerous tile sizes.

\section*{Tilings with four or more tile sizes}
We first consider tilings with four tile sizes: $T$ consists
of $n_1\geq 2$ tiles with side length $1/a$ and $n_i$ tiles
with side length $m_i/a$, $2\leq i\leq 4$, with
$m_4>m_3>m_2>1$. In this case, $\delta(T)$ takes
the form
\begin{align*}
\delta(T) &= A_2+A_3+A_4,\\
\text{where }A_2&= n_2\big[n_1(m_2-1)^2 -m_2^2+n_3(m_3-m_2)^2+n_4(m_4-m_2)^2\big],\\
A_3&= n_3\big[n_1(m_3-1)^2 -m_3^2+n_4(m_4-m_3)^2\big],\\
A_4 &= n_4\big[n_1(m_4-1)^2-m_4^2\big]-n_1.
\end{align*}

\begin{lemma}\label{l4}
If $T$ has four or more tile sizes and $n_1\geq 2$, then
$\delta(T)\geq 0$.
\end{lemma}
\begin{proof}
We prove this for four tile sizes. The proof for five or more tile
sizes is similar. We show that $\delta(T)\geq 0$ in all
cases.

First note that $m_4\geq 4$, so by Lemma~\ref{n1}
we have $A_4\geq 0$. Second, note that
$A_3\geq  n_3\big[2(m_3-1)^2 -m_3^2+n_4(m_4-m_3)^2\big]
\geq (m_3-2)^2-2+n_4(m_4-m_3)^2$; since $m_3\geq 3$,
we have $A_3\geq 0$. Finally, we have
$A_2\geq (m_2-2)^2-2+n_3(m_3-m_2)^2+n_4(m_4-m_2)^2
\geq 0$. Thus $\delta(T)\geq 0$ as required.
\end{proof}

Combining Lemmas \ref{l2}, \ref{l3}, and \ref{l4}, we get
the main result of this note.

\begin{theorem*}
If $T$ is an MTP tiling with $k^2+3$ tiles and $\sigma(T)
>k+1/k$, then $T$ consists of at least three tile sizes 
and $T$ has a unique smallest tile. 
\end{theorem*}

So to show that $f_M(k^2+3)=k+1/k$, we only need to
investigate MTP tilings with a unique smallest tile.

\end{document}